\documentclass[11pt]{amsart}%
\usepackage{amsmath}
\usepackage{graphicx}
\usepackage{amsfonts}
\usepackage{amssymb}%
\newtheorem{theorem}{Theorem}[section]
\theoremstyle{plain}

\newtheorem{lemma}{Lemma}[section]

\newtheorem{proposition}{Proposition}[section]

\numberwithin{equation}{section}
\begin{document}
\title[Hessian Estimates]{Hessian estimates for special Lagrangian equations with critical and
supercritical phases in general dimensions}
\author{Dake WANG}
\author{Yu YUAN}
\address{Department of Mathematics, Box 354350\\
University of Washington\\
Seattle, WA 98195}
\email{dkpool@math.washington.edu, yuan@math.washington.edu}
\thanks{Both authors are partially supported by an NSF grant.}
\date{\today}

\begin{abstract}
We derive a priori interior Hessian estimates for special Lagrangian equation
with critical and supercritical phases in general higher dimensions. Our
unified approach leads to sharper estimates even for the previously known
three dimensional and convex solution cases.

\end{abstract}
\maketitle

\section{\bigskip Introduction}

In this paper, we complete a priori\emph{ interior} Hessian estimates for the
special Lagrangian equation%

\begin{equation}
\sum_{i=1}^{n}\arctan\lambda_{i}=\Theta\label{EsLag}%
\end{equation}
with \emph{critical} and \emph{supercritical} phases $\left\vert
\Theta\right\vert \geq\left(  n-2\right)  \pi/2$ in all dimensions $n\geq3,$
where $\lambda=\left(  \lambda_{1},\cdots,\lambda_{n}\right)  $ are the
eigenvalues of the Hessian $D^{2}u.$ For solutions to $(\ref{EsLag})$ with
$\left\vert \Theta\right\vert \geq\left(  n-2\right)  \pi/2$ in dimension two
and three, and also convex solutions to $(\ref{EsLag})$ in all dimensions,
Hessian estimates have been obtained in [WY2,3,4] and [CWY].

Equation (\ref{EsLag}) originates in the special Lagrangian geometry by
Harvey-Lawson [HL]. The Lagrangian graph $\left(  x,Du\left(  x\right)
\right)  \subset\mathbb{R}^{n}\times\mathbb{R}^{n}$ is called special when the
argument of the complex number $\left(  1+\sqrt{-1}\lambda_{1}\right)
\cdots\left(  1+\sqrt{-1}\lambda_{n}\right)  $ or the phase is constant
$\Theta,$ and it is special if and only if $\left(  x,Du\left(  x\right)
\right)  $ is a (volume minimizing) minimal surface in $\mathbb{R}^{n}%
\times\mathbb{R}^{n}$ [HL, Theorem 2.3, Proposition 2.17]. The phase $\left(
n-2\right)  \pi/2$ is called critical because the level set $\left\{  \left.
\lambda\in\mathbb{R}^{n}\right\vert \lambda\ \text{satisfying }(\ref{EsLag}%
)\right\}  $ is convex \emph{only} when $\left\vert \Theta\right\vert
\geq\left(  n-2\right)  \pi/2$ [Y2, Lemma 2.1]. The algebraic form of
(\ref{EsLag}) is%
\begin{equation}
\cos\Theta\sum_{1\leq2k+1\leq n}\left(  -1\right)  ^{k}\sigma_{2k+1}%
-\sin\Theta\sum_{0\leq2k\leq n}\left(  -1\right)  ^{k}\sigma_{2k}=0,
\label{EsLagA}%
\end{equation}
where $\sigma_{k}$s are the elementary symmetric functions of the Hessian
$D^{2}u.$

We state our main result in the following.

\begin{theorem}
Let $u$ be a smooth solution to (\ref{EsLag}) with $\left\vert \Theta
\right\vert \geq\left(  n-2\right)  \pi/2$ and $n\geq3$ on $B_{R}%
(0)\subset\mathbb{R}^{n}.$ Then we have
\[
|D^{2}u(0)|\leq C(n)\exp\left[  C(n)\max_{B_{R}(0)}|Du|^{2n-2}/R^{2n-2}%
\right]  ;
\]
and when $\left\vert \Theta\right\vert =\left(  n-2\right)  \pi/2,$ we also
have%
\[
|D^{2}u(0)|\leq C(n)\exp\left[  C(n)\max_{B_{R}(0)}|Du|^{2n-4}/R^{2n-4}%
\right]  .
\]

\end{theorem}

Relying on our previous gradient estimates for (\ref{EsLag}) with $\left\vert
\Theta\right\vert \geq\left(  n-2\right)  \pi/2$ in [WY4]%
\[
\max_{B_{R}(0)}|Du|\leq C\left(  n\right)  \left[  \operatorname*{osc}%
_{B_{2R}\left(  0\right)  }\frac{u}{R}+1\right]  ,
\]
we can bound $D^{2}u$ in terms of the solution $u$ in $B_{2R}\left(  0\right)
.$

Singular (viscosity) solutions to (\ref{EsLag}) with \emph{subcritical} phases
$\left\vert \Theta\right\vert <\left(  n-2\right)  \pi/2$ and $n\geq3$
constructed by Nadirashvili-Vl\u{a}du\c{t} [NV] and the authors [WdY] show
that the critical and supercritical phase condition in Theorem 1.1 is necessary.

One application of the above estimates is the regularity (analyticity) of the
$C^{0}$ viscosity solutions to (\ref{EsLag}) with $\left\vert \Theta
\right\vert \geq\left(  n-2\right)  \pi/2.$ In particular, the solutions of
the Dirichlet problem with continuous boundary data to (\ref{EsLag}) with
convex condition $\left\vert \Theta\right\vert \geq\left(  n-2\right)  \pi/2$
enjoy interior regularity. In contrast, the Hessian estimates, then the
interior regularity for solutions to (\ref{EsLag}) with $\left\vert
\Theta\right\vert =\left[  \frac{n-1}{2}\right]  \pi$ in [CNS] by
Caffarelli-Nirenberg-Spruck was derived under the $C^{4}$ smoothness
assumption on the boundary data.

Another quick consequence is a Liouville type result for global solutions with
quadratic growth to (\ref{EsLag}) with $\left\vert \Theta\right\vert =\left(
n-2\right)  \pi/2,$ namely any such a solution must be quadratic (cf. [Y1],
[Y2] where other Liouville type results for convex solutions to (\ref{EsLag})
and Bernstein type results for global solutions to (\ref{EsLag}) with
supercritical phase $\left\vert \Theta\right\vert >\left(  n-2\right)  \pi/2$
were obtained).

In the 1950's, Heinz [H] derived a Hessian bound for the two dimensional
Monge-Amp\`{e}re type equation including (\ref{EsLag}) with $n=2;$ see also
Pogorelov [P1] for Hessian estimates for these equations including
(\ref{EsLag}) with $\left\vert \Theta\right\vert >\pi/2\ $and $n=2.$ In the
1970's Pogorelov [P2] constructed his famous counterexamples, namely irregular
solutions to three dimensional Monge-Amp\`{e}re equations $\sigma_{3}%
(D^{2}u)=\det(D^{2}u)=1;$ those irregular solutions also serve as
counterexamples for cubic and higher order symmetric $\sigma_{k}$ equations
(cf. [U1]). In passing, we also mention Hessian estimates for solutions with
certain \emph{strict} convexity constraints to Monge-Amp\`{e}re equations and
$\sigma_{k}$ equation ($k\geq2$) by Pogorelov [P2] and Chou-Wang [CW]
respectively using the Pogorelov technique. Trudinger [T2] and Urbas [U2][U3],
also Bao-Chen [BC] obtained (pointwise) Hessian estimates in terms of certain
integrals of the Hessian, for $\sigma_{k}$ equations and special Lagrangian
equation (1.1) with $n=3,\ \Theta=\pi$ respectively. Pointwise Hessian
estimates for strictly convex solutions to quotient equations $\sigma
_{n}/\sigma_{k}$ were derived in terms of certain integrals of the Hessian by
Bao-Chen-Guan-Ji [BCGJ].

Our strategies for the Hessian estimates go as follows. We bound the
subharmonic function of the Hessian $b=\ln\sqrt{1+\lambda_{\max}^{2}}$ by its
integral on the minimal surface using Michael-Simon's mean value inequality
[MS]. Applying certain Sobolev inequalities, we estimate the integral of $b$
by the integral of its gradient. The decisive choice $b$ satisfies a Jacobi
inequality: its Laplacian bounds its gradient; in turn, the integral of the
gradient $b$ is bounded by a weighted volume of the minimal Lagrangian graph.
By a conformality identity, the weighted volume element is in fact the trace
of the linearized operator of the special Lagrangian equation in algebraic
form, which is a linear combination of the elementary symmetric functions of
the Hessian. Taking advantage of the divergence structure of those functions,
we bound the weighted volume in terms of the height of special Lagrangian
graph, or the gradient of the solution.

However, there are two major difficulties in the execution for general
dimension. The first one is to justify the nonlinear Jacobi inequality in the
integral sense for the Lipschitz only function $b,$ which was only achieved in
dimension three by involved arguments [WY2]. The second one is to find, in the
critical phase case, a relative isoperimetric inequality or equivalent Sobolev
inequality for functions without compact support, which was circumvented only
in dimension three thanks to the linear dependence on the Hessian for the
linearized operator of now equivalent equation $\sigma_{2}=1$ [WY2]. We
overcome the first one by observing that the Jacobi inequality and its
equivalent linear formulation hold in the viscosity sense, consequently in the
potential sense. By Herv\'{e}-Herv\'{e} [HH, Theorem 1] (see also Watson [Wn,
p. 246]), the linear inequality holds in the integral sense, in turn, so does
the needed Jacobi inequality. Conceptually it is natural this way. For
details, see the proof of Proposition 2.1. To deal with the second difficulty,
we instead apply the Sobolev inequality for functions with compact supports,
but use a \textquotedblleft twist-multiplication\textquotedblright\ trick to
contain the terms involving derivatives of the cut-off functions (Step 4 in
Section 3). This trick enables us to have a unified approach (for both the
critical and supercritical cases) in all dimensions $n\geq3.$ Even in the
known three dimensional [WY2,4] and convex cases [CWY], the simpler unified
argument leads to sharper Hessian estimates.

Our unified arguments does not work for (\ref{EsLag}) with $\Theta=0$ and
$n=2,$ as the Jacobi inequality fails (only) for harmonic functions.
Elementary methods in [WY3] led to the sharp Hessian estimates in dimension
two. (The sharp Hessian estimates in terms of the linear exponential
dependence on the gradients, can be seen by the corresponding solutions to the
Monge-Amp\`{e}re equation or (\ref{EsLag}) with $\Theta=\pi/2$ and $n=2,$
converted from Finn's minimal surface [F, p. 355] via Heinz transformation [J,
p. 133].)

As one can see that, not only our Hessian-slope estimates for
\textquotedblleft gradient\textquotedblright\ minimal graphs are analogous to
the gradient-slope estimates for the codimension one minimal graphs, but also
our arguments resemble the original integral proof by
Bombieri-De\ Giorgi-Miranda [BDM] and the simplified one by Trudinger [T1] for
the latter classical result. When one tries to adapt the later Korevaar
pointwise technique [K], certain extra structure or assumption has to be used,
as in [WY1]. Otherwise, an adaptation of the technique alone would lead to
Hessian estimates for the Monge-Amp\`{e}re equations, to which the Jacobi
inequality is available. But this is inconsistent with Pogorelov's singular
solutions [P2].

\textbf{Notation. }First\textbf{ }$\partial_{i}=\frac{\partial}{\partial
_{x_{i}}},\ \partial_{ij}=\frac{\partial^{2}}{\partial x_{i}\partial x_{j}%
},\ u_{i}=\partial_{i}u=D_{i}u,\ u_{ji}=\partial_{ij}u$ etc., but $\lambda
_{1},\cdots,\lambda_{n}$ and $b_{k}=\left(  \ln\sqrt{1+\lambda_{1}^{2}}%
+\cdots+\ln\sqrt{1+\lambda_{k}^{2}}\right)  /k\ $do not represent the partial
derivatives. Also%
\[
\sigma_{k}\left(  \lambda_{1},\cdots,\lambda_{n}\right)  =\sum_{1\leq
i_{1}<\cdots<i_{k}\leq n}\lambda_{i_{1}}\cdots\lambda_{i_{k}}.
\]
Further, $h_{ijk}$ will denote (the second fundamental form)
\[
h_{ijk}=\frac{1}{\sqrt{1+\lambda_{i}^{2}}}\frac{1}{\sqrt{1+\lambda_{j}^{2}}%
}\frac{1}{\sqrt{1+\lambda_{k}^{2}}}u_{ijk}.
\]
when $D^{2}u$ is diagonalized. Finally $C\left(  n\right)  $ will denote
various constants depending only on dimension $n.$

\section{Preliminary inequalities}

Taking the gradient of both sides of the special Lagrangian equation
(\ref{EsLag}), we have
\begin{equation}
\sum_{i,j=1}^{n}g^{ij}\partial_{ij}\left(  x,Du\left(  x\right)  \right)  =0,
\label{Emin}%
\end{equation}
where $\left(  g^{ij}\right)  $ is the inverse of the induced metric
$g=\left(  g_{ij}\right)  =I+D^{2}uD^{2}u$ on the surface $\left(  x,Du\left(
x\right)  \right)  \subset\mathbb{R}^{n}\times\mathbb{R}^{n}.$ Simple
geometric manipulation of (\ref{Emin}) yields the usual form of the minimal
surface equation
\[
\bigtriangleup_{g}\left(  x,Du\left(  x\right)  \right)  =0,
\]
where the Laplace-Beltrami operator of the metric $g$ is given by
\[
\bigtriangleup_{g}=\frac{1}{\sqrt{\det g}}\sum_{i,j=1}^{n}\partial_{i}\left(
\sqrt{\det g}g^{ij}\partial_{j}\right)  .
\]
Because we are using harmonic coordinates $\bigtriangleup_{g}x=0,$ we see that
$\bigtriangleup_{g}$ also equals the linearized operator of the special
Lagrangian equation (\ref{EsLag}) at $u,$%
\[
\bigtriangleup_{g}=\sum_{i,j=1}^{n}g^{ij}\partial_{ij}.
\]
The volume form, gradient and inner product with respect to the metric $g$
are
\begin{align*}
dv_{g}  &  =\sqrt{\det g}\ dx,\\
\nabla_{g}v  &  =\left(  \sum_{k=1}^{n}g^{1k}v_{k},\cdots,\sum_{k=1}^{n}%
g^{nk}v_{k}\right)  ,\\
\left\langle \nabla_{g}v,\nabla_{g}w\right\rangle _{g}  &  =\sum_{i,j=1}%
^{n}g^{ij}v_{i}w_{j},\ \ \text{in particular \ }\left\vert \nabla
_{g}v\right\vert ^{2}=\left\langle \nabla_{g}v,\nabla_{g}v\right\rangle _{g}.
\end{align*}
We begin with some algebraic and trigonometric inequalities needed in this paper.

\begin{lemma}
Suppose the ordered real numbers $\lambda_{1}\geq\lambda_{2}\geq\cdots
\geq\lambda_{n}$ satisfy (\ref{EsLag}) with $\Theta\geq\left(  n-2\right)
\pi/2$ and $n\geq2.$ Then we have%
\begin{gather}
\lambda_{1}\geq\lambda_{2}\geq\cdots\geq\lambda_{n-1}>0\ \ \text{and }%
\lambda_{n-1}\geq\left\vert \lambda_{n}\right\vert ,\label{AT-allbut1+}\\
\lambda_{1}+\left(  n-1\right)  \lambda_{n}\geq0,\label{AT-maxmin}\\
\sigma_{k}\left(  \lambda_{1},\cdots,\lambda_{n}\right)  \geq0\ \ \text{for
all }1\leq k\leq n-1. \label{AT-sigma}%
\end{gather}

\end{lemma}

\begin{proof}
Set $\theta_{i}=\arctan\lambda_{i}.$ Property (\ref{AT-allbut1+}) follows from
the inequalities
\[
\theta_{n-1}+\theta_{n}\geq\left(  n-2\right)  \pi/2-\left(  \theta_{1}%
+\cdots+\theta_{n-2}\right)  \geq0.
\]

We only need to check property (\ref{AT-maxmin}) when $\lambda_{n}<0$ or
$\theta_{n}<0.$ We know%
\[
\frac{\pi}{2}>\frac{\pi}{2}+\theta_{n}\geq\left(  \frac{\pi}{2}-\theta
_{1}\right)  +\cdots+\left(  \frac{\pi}{2}-\theta_{n-1}\right)  >0.
\]
It follows that%
\begin{align}
-\frac{1}{\lambda_{n}}  &  =\tan\left(  \frac{\pi}{2}+\theta_{n}\right)
\label{AT-reciprocal}\\
&  \geq\tan\left(  \frac{\pi}{2}-\theta_{1}\right)  +\cdots+\tan\left(
\frac{\pi}{2}-\theta_{n-1}\right)  =\frac{1}{\lambda_{1}}+\cdots+\frac
{1}{\lambda_{n-1}}\nonumber\\
&  \geq\left(  n-1\right)  \frac{1}{\lambda_{1}}.\nonumber
\end{align}
Then we get (\ref{AT-maxmin}).

Next we prove property (\ref{AT-sigma}) with $k=n-1.$ We only need to deal
with the case $\lambda_{n}<0.$ From (\ref{AT-reciprocal}), we have%
\[
0\geq\frac{1}{\lambda_{1}}+\cdots+\frac{1}{\lambda_{n-1}}+\frac{1}{\lambda
_{n}}=\frac{\sigma_{n-1}\left(  \lambda_{1},\cdots,\lambda_{n}\right)
}{\left(  \lambda_{1}\cdots\lambda_{n-1}\right)  \lambda_{n}}.
\]
Using $\lambda_{1}\geq\lambda_{2}\geq\cdots\geq\lambda_{n-1}>0>\lambda_{n},$
we get $\sigma_{n-1}\left(  \lambda_{1},\cdots,\lambda_{n}\right)  \geq0.$

Finally we prove the whole property (\ref{AT-sigma}) inductively. Property
(\ref{AT-sigma}) with $n=2$ is obvious (or by the above). Assume property
(\ref{AT-sigma}) with $n=m$ is true, that is%
\[
\sigma_{j}\left(  \lambda_{1},\cdots,\lambda_{m}\right)  \geq0\ \ \ \text{for
}1\leq j\leq m-1,\ \ \
\]
provided $\arctan\lambda_{1}+\cdots+\arctan\lambda_{m}\geq\left(  m-2\right)
\pi/2.$

Let us prove (\ref{AT-sigma}) with $n=m+1$ for
\begin{equation}
\arctan\lambda_{1}+\cdots+\arctan\lambda_{m+1}\geq\left(  m-1\right)  \pi/2.
\label{critical m+1}%
\end{equation}

By the proved property (\ref{AT-sigma}) with $k=n-1=m,$ we get $\sigma
_{m}\left(  \lambda_{1},\cdots,\lambda_{m+1}\right)  \geq0.$ We only need to
verify the other $\sigma$ inequalities when the smallest number is negative,
say $\lambda_{1}\geq\lambda_{2}\geq\cdots\geq\lambda_{m}>0>\lambda_{m+1}.$ (By
(\ref{AT-allbut1+}), only the smallest $\lambda_{m+1}$ can be negative.) We
have%
\[
\sigma_{m-1}\left(  \lambda_{1},\cdots,\lambda_{m+1}\right)  =\sigma
_{m-1}\left(  \lambda_{2},\cdots,\lambda_{m+1}\right)  +\lambda_{1}%
\sigma_{m-2}\left(  \lambda_{2},\cdots,\lambda_{m+1}\right)  .
\]
From (\ref{critical m+1}), we infer%
\[
\arctan\lambda_{2}+\cdots+\arctan\lambda_{m+1}\geq\left(  m-2\right)  \pi/2.
\]
By the induction assumption, we should have
\[
\sigma_{m-1}\left(  \lambda_{2},\cdots,\lambda_{m+1}\right)  \geq
0\ \text{and}\ \sigma_{m-2}\left(  \lambda_{2},\cdots,\lambda_{m+1}\right)
\geq0.
\]
Thus we obtain $\sigma_{m-1}\left(  \lambda_{1},\cdots,\lambda_{m+1}\right)
\geq0.$ Similarly we prove $\sigma_{i}\left(  \lambda_{1},\cdots,\lambda
_{m+1}\right)  \geq0$ for $1\leq i\leq m-2.$ Therefore property
(\ref{AT-sigma}) holds for all $n\geq2.$ This completes the proof of Lemma 2.1.
\end{proof}

\begin{lemma}
Let $u$ be a smooth solution to (\ref{EsLag}). Suppose that the Hessian
$D^{2}u$ is diagonalized and the eigenvalue $\lambda_{\gamma}$ is distinct
from all other eigenvalues of $D^{2}u$ at point $p.$ Then we have at $p$%
\begin{equation}
\left\vert \nabla_{g}\ln\sqrt{1+\lambda_{\gamma}^{2}}\right\vert ^{2}%
=\sum_{k=1}^{n}\lambda_{\gamma}^{2}h_{\gamma\gamma k}^{2} \label{gradientb1}%
\end{equation}
and%
\begin{align}
&  \bigtriangleup_{g}\ln\sqrt{1+\lambda_{\gamma}^{2}}\overset{}{=}%
\label{lapb1}\\
&  (1+\lambda_{\gamma}^{2})h_{\gamma\gamma\gamma}^{2}+\sum_{k\neq\gamma
}\left(  \frac{2\lambda_{\gamma}}{\lambda_{\gamma}-\lambda_{k}}+\frac
{2\lambda_{\gamma}^{2}\lambda_{k}}{\lambda_{\gamma}-\lambda_{k}}\right)
h_{kk\gamma}^{2}\nonumber\\
&  +\sum_{k\neq\gamma}\left[  1+\frac{2\lambda_{\gamma}}{\lambda_{\gamma
}-\lambda_{k}}+\frac{\lambda_{\gamma}^{2}\left(  \lambda_{\gamma}+\lambda
_{k}\right)  }{\lambda_{\gamma}-\lambda_{k}}\right]  h_{\gamma\gamma k}%
^{2}\nonumber\\
&  +\sum_{\substack{k>j\\k,\ j\neq\gamma}}2\lambda_{\gamma}\left[
\frac{1+\lambda_{k}^{2}}{\lambda_{\gamma}-\lambda_{k}}+\frac{1+\lambda_{j}%
^{2}}{\lambda_{\gamma}-\lambda_{j}}+(\lambda_{j}+\lambda_{k})\right]
h_{kj\gamma}^{2}.\nonumber
\end{align}

\end{lemma}

\begin{proof}
The calculation was done in Lemma 2.1 of [WY2].
\end{proof}

\begin{lemma}
Let $u$ be a smooth solution to (\ref{EsLag}) with $\Theta\geq(n-2)\frac{\pi
}{2}.$ Suppose that the ordered eigenvalues $\lambda_{1}\geq\lambda_{2}%
\geq\cdots\geq\lambda_{n}$ of the Hessian $D^{2}u$ satisfy $\lambda_{1}%
=\cdots=\lambda_{m}>\lambda_{m+1}$ at point $p.$ Then the function
$b_{m}=\frac{1}{m}\sum_{i=1}^{m}\ln\sqrt{1+\lambda_{i}^{2}}$\ is smooth near
$p$ and satisfies at $p$%
\begin{equation}
\bigtriangleup_{g}b_{m}\geq\left(  1-\frac{4}{\sqrt{4n+1}+1}\right)
\left\vert \nabla_{g}b_{m}\right\vert ^{2}. \label{jacobi-bm}%
\end{equation}

\end{lemma}

\begin{proof}
Step 1. The function $b_{m}$ is symmetric in $\lambda_{1},\cdots,\lambda_{m}.$
Thus\ for $m<n,$ $b_{m}$ is smooth when $\lambda_{m}>\lambda_{m+1},$ in
particular near $p,$ at which $\lambda_{1}=\cdots=\lambda_{m}>\lambda_{m+1}.$
For $m=n,$ $b_{n}$ is certainly smooth everywhere.

We again assume that Hessian $D^{2}u$ is diagonalized at point $p.$ Let us
also first assume the first $m$ eigenvalues $\lambda_{1,}\cdots,\lambda_{m}$
are distinct. Using (\ref{lapb1}) in Lemma 2.2, we calculate $\bigtriangleup
_{g}b_{m};$ after grouping those terms $h_{\heartsuit\heartsuit\heartsuit},$
$h_{\heartsuit\heartsuit\clubsuit}$ and $h_{\heartsuit\clubsuit\diamondsuit}$
in the summation, we obtain%
\begin{align*}
&  m\bigtriangleup_{g}b_{m}\overset{}{=}\sum_{\gamma=1}^{m}\bigtriangleup
_{g}\ln\sqrt{1+\lambda_{\gamma}^{2}}\overset{p}{=}\\
&
{\textstyle\sum\limits_{k\leq m}}
\left(  1+\lambda_{k}^{2}\right)  {\small h}_{kkk}^{2}{\small +}(%
{\textstyle\sum\limits_{i<k\leq m}}
+%
{\textstyle\sum\limits_{k<i\leq m}}
)\left(  3+\lambda_{i}^{2}+2\lambda_{i}\lambda_{k}\right)  {\small h}%
_{iik}^{2}{\small +}%
{\textstyle\sum\limits_{k\leq m<i}}
\frac{2\lambda_{k}\left(  1+\lambda_{k}\lambda_{i}\right)  }{\lambda
_{k}-\lambda_{i}}{\small h}_{iik}^{2}\\
&  +\sum_{i\leq m<k}\frac{3\lambda_{i}-\lambda_{k}+\lambda_{i}^{2}\left(
\lambda_{i}+\lambda_{k}\right)  }{\lambda_{i}-\lambda_{k}}h_{iik}^{2}+\\
&  \left\{
\begin{array}
[c]{l}%
2\sum_{i<j<k\leq m}\left(  3+\lambda_{i}\lambda_{j}+\lambda_{j}\lambda
_{k}+\lambda_{k}\lambda_{i}\right)  h_{ijk}^{2}+\\
2\sum_{i<j\leq m<k}\left(  1+\lambda_{i}\lambda_{j}+\lambda_{j}\lambda
_{k}+\lambda_{k}\lambda_{i}+\lambda_{i}\frac{1+\lambda_{k}^{2}}{\lambda
_{i}-\lambda_{k}}+\lambda_{j}\frac{1+\lambda_{k}^{2}}{\lambda_{j}-\lambda_{k}%
}\right)  h_{ijk}^{2}+\\
2\sum_{i\leq m<j<k}\lambda_{i}\left[  \lambda_{j}+\lambda_{k}+\frac
{1+\lambda_{j}^{2}}{\lambda_{i}-\lambda_{j}}+\frac{1+\lambda_{k}^{2}}%
{\lambda_{j}-\lambda_{k}}\right]  h_{ijk}^{2}%
\end{array}
\right.  .
\end{align*}

Now as a function of the matrices (then composed with smooth matrix function
$D^{2}u$ of $x$), $b_{m}$ is $C^{2}$ at $D^{2}u\left(  p\right)  $ with
eigenvalues satisfying $\lambda=\lambda_{1}=\cdots=\lambda_{m}>\lambda_{m+1}.$
Note that $D^{2}u\left(  p\right)  $ can be approximated by matrices with
distinct eigenvalues. Therefore the above expression for $\bigtriangleup
_{g}b_{m}$ at $p$ still holds and simplifies to%
\begin{align*}
&  m\bigtriangleup_{g}b_{m}\overset{p}{=}\\
&
{\textstyle\sum\limits_{k\leq m}}
\left(  1+\lambda^{2}\right)  h_{kkk}^{2}+(%
{\textstyle\sum\limits_{i<k\leq m}}
+%
{\textstyle\sum\limits_{k<i\leq m}}
)\left(  3+3\lambda^{2}\right)  h_{iik}^{2}+%
{\textstyle\sum\limits_{k\leq m<i}}
\frac{2\lambda\left(  1+\lambda\lambda_{i}\right)  }{\lambda-\lambda_{i}%
}h_{iik}^{2}+\\
&  \sum_{i\leq m<k}\frac{3\lambda-\lambda_{k}+\lambda^{2}\left(
\lambda+\lambda_{k}\right)  }{\lambda-\lambda_{k}}h_{iik}^{2}+\\
&  \left\{
\begin{array}
[c]{l}%
2\sum_{i<j<k\leq m}\left(  3+3\lambda^{2}\right)  h_{ijk}^{2}+\\
2\sum_{i<j\leq m<k}\left[  1+\frac{2\lambda}{\lambda-\lambda_{k}}%
+\frac{\lambda^{2}\left(  \lambda+\lambda_{k}\right)  }{\lambda-\lambda_{k}%
}\right]  h_{ijk}^{2}+\\
2\sum_{i\leq m<j<k}\lambda\left(  \lambda_{j}+\lambda_{k}+\frac{1+\lambda
_{j}^{2}}{\lambda-\lambda_{j}}+\frac{1+\lambda_{k}^{2}}{\lambda-\lambda_{k}%
}\right)  h_{ijk}^{2}%
\end{array}
\right. \\
&  \geq\sum_{k\leq m}\lambda^{2}h_{kkk}^{2}+(%
{\textstyle\sum_{i<k\leq m}}
+%
{\textstyle\sum_{k<i\leq m}}
)3\lambda^{2}h_{iik}^{2}+\sum_{k\leq m<i}\frac{2\lambda^{2}\lambda_{i}%
}{\lambda-\lambda_{i}}h_{iik}^{2}+\\
&  \sum_{i\leq m<k}\frac{\lambda^{2}\left(  \lambda+\lambda_{k}\right)
}{\lambda-\lambda_{k}}h_{iik}^{2},
\end{align*}
where we used (\ref{AT-allbut1+}) of Lemma 2.1 in the inequality.

Similarly by (\ref{gradientb1}) in Lemma 2.2 and the $C^{1}$ continuity of
$b_{m}$ as a function of matrices at $D^{2}u\left(  p\right)  ,$ we obtain
\[
\left\vert \nabla_{g}b_{m}\right\vert ^{2}\overset{p}{=}\frac{1}{m^{2}}%
\sum_{1\leq k\leq n}\lambda^{2}\left(  \sum_{i\leq m}h_{iik}\right)  ^{2}%
\leq\frac{\lambda^{2}}{m}\sum_{1\leq k\leq n}\left(  \sum_{i\leq m}h_{iik}%
^{2}\right)  .
\]

From the above two inequalities, it follows that%
\begin{gather}
m\left(  \bigtriangleup_{g}b_{m}-\varepsilon\left\vert \nabla_{g}%
b_{m}\right\vert ^{2}\right)  \geq\nonumber\\
\lambda^{2}\left[
\begin{array}
[c]{c}%
{\textstyle\sum_{k\leq m}}
\left(  1-\varepsilon\right)  h_{kkk}^{2}+(%
{\textstyle\sum_{i<k\leq m}}
+%
{\textstyle\sum_{k<i\leq m}}
)\left(  3-\varepsilon\right)  h_{iik}^{2}\\
+%
{\textstyle\sum_{k\leq m<i}}
\frac{2\lambda_{i}}{\lambda-\lambda_{i}}h_{iik}^{2}%
\end{array}
\right]  +\label{I}\\
\lambda^{2}\left[  \sum_{i\leq m<k}\left(  \frac{\lambda+\lambda_{k}}%
{\lambda-\lambda_{k}}-\varepsilon\right)  h_{iik}^{2}\right]  \label{II}%
\end{gather}
with $\varepsilon$ to be fixed.

Step 2. We show (\ref{I}) and (\ref{II}) in the above inequality are
nonnegative for $\varepsilon=1-4/\left(  \sqrt{4n+1}+1\right)  .$ For each
fixed $k$ in (\ref{I}) and (\ref{II}), set $t_{i}=h_{iik}.$ By the minimal
surface equation (\ref{Emin}), we have%
\begin{equation}
t_{1}+\cdots+t_{n}=0. \label{Emin-t}%
\end{equation}

Step 2.1. For each fixed $k\leq m,$ we prove the $\left[  \ \right]  _{k}$
term in (\ref{I}) is nonnegative. In the case with all $\lambda_{i}\geq0,$ the
nonnegativity is straightforward. In the remaining worst case $\lambda
_{n-1}>0>\lambda_{n}.$ Without loss of generality, we assume $k=1$ for simple
notation. Then we proceed as follows:%
\begin{gather*}
\left[  \ \right]  _{1}=\left\{  \left(  1-\varepsilon\right)  t_{1}^{2}%
+\sum_{i=2}^{m}\left(  3-\varepsilon\right)  t_{i}^{2}+\sum_{i=m+1}^{n-1}%
\frac{2\lambda_{i}}{\lambda-\lambda_{i}}t_{i}^{2}\right\}  +\frac{2\lambda
_{n}}{\lambda-\lambda_{n}}t_{n}^{2}\\
=\left\{  \left(  1-\varepsilon\right)  t_{1}^{2}+\sum_{i=2}^{m}\left(
3-\varepsilon\right)  t_{i}^{2}+\sum_{i=m+1}^{n-1}\frac{2\lambda_{i}}%
{\lambda-\lambda_{i}}t_{i}^{2}\right\}  +\frac{2\lambda_{n}}{\lambda
-\lambda_{n}}\left(  \sum_{i=1}^{n-1}t_{i}\right)  ^{2}\\
\geq\left\{  \left(  1-\varepsilon\right)  t_{1}^{2}+\sum_{i=2}^{m}\left(
3-\varepsilon\right)  t_{i}^{2}+\sum_{i=m+1}^{n-1}\frac{2\lambda_{i}}%
{\lambda-\lambda_{i}}t_{i}^{2}\right\}  \cdot\\
\left[  1+\frac{2\lambda_{n}}{\lambda-\lambda_{n}}\left(  \frac{1}%
{1-\varepsilon}+\sum_{i=2}^{m}\frac{1}{3-\varepsilon}+\sum_{i=m+1}^{n-1}%
\frac{\lambda-\lambda_{i}}{2\lambda_{i}}\right)  \right]  ,
\end{gather*}
where we used (\ref{Emin-t}) and a Cauchy-Schartz inequality to reach the
above inequality. We now show the second factor $\left[  \ \right]  $ in the
last term is also nonnegative:%
\begin{align*}
&  \left[  1+\frac{2\lambda_{n}}{\lambda-\lambda_{n}}\left(  \frac
{1}{1-\varepsilon}+\sum_{i=2}^{m}\frac{1}{3-\varepsilon}+\sum_{i=m+1}%
^{n-1}\frac{\lambda-\lambda_{i}}{2\lambda_{i}}\right)  \right] \\
&  =\frac{2\lambda_{n}}{\lambda-\lambda_{n}}\left(  \frac{\lambda-\lambda_{n}%
}{2\lambda_{n}}+\frac{1}{1-\varepsilon}+\frac{m-1}{3-\varepsilon}%
+\frac{\lambda-\lambda_{m+1}}{2\lambda_{m+1}}+\cdots+\frac{\lambda
-\lambda_{n-1}}{2\lambda_{n-1}}\right) \\
&  =\frac{2\lambda_{n}}{\lambda-\lambda_{n}}\left[  \frac{1}{1-\varepsilon
}+\frac{m-1}{3-\varepsilon}+\frac{\lambda}{2}\left(  \frac{1}{\lambda_{1}%
}+\cdots+\frac{1}{\lambda_{1}}\right)  -\frac{n}{2}\right] \\
&  =\frac{2\lambda_{n}}{\lambda-\lambda_{n}}\left[  \frac{1}{1-\varepsilon
}+\frac{m-1}{3-\varepsilon}+\frac{\lambda}{2}\frac{\sigma_{n-1}}{\sigma
n}-\frac{n}{2}\right] \\
&  \geq\frac{2\lambda_{n}}{\lambda-\lambda_{n}}\left(  \frac{1}{1-\varepsilon
}+\frac{m-1}{3-\varepsilon}-\frac{n}{2}\right) \\
&  \geq0,
\end{align*}
where we used $\lambda_{1}=\cdots=\lambda_{m},$ (\ref{AT-sigma}), and
$\frac{1}{1-\varepsilon}+\frac{m-1}{3-\varepsilon}-\frac{n}{2}\leq0$ under the
assumption
\[
\varepsilon\leq2-\frac{m}{n}-\sqrt{\left(  1-\frac{m}{n}\right)  ^{2}+\frac
{4}{n}}.
\]
Therefor $[\ ]_{1}\geq0.$

Step 2.2. For each $k$ between $m$ and $n,$ we have $\lambda_{k}>0,$ the
$\left[  \ \right]  _{k}$ term in (\ref{II}) satisfies%
\begin{align*}
\left[  \ \right]  _{k}  &  =\sum_{i\leq m}\left(  \frac{\lambda+\lambda_{k}%
}{\lambda-\lambda_{k}}-\varepsilon\right)  t_{i}^{2}\\
&  \geq\sum_{i\leq m}\left(  1-\varepsilon\right)  t_{i}^{2}\geq0,
\end{align*}
as long as $\varepsilon\leq1.$

For $k=n,$ the $\left[  \ \right]  _{n}$ term in (\ref{II}) becomes%
\begin{align*}
\left[  \ \right]  _{n}  &  =\sum_{i\leq m}\left(  \frac{\lambda+\lambda_{n}%
}{\lambda-\lambda_{n}}-\varepsilon\right)  t_{i}^{2}\\
&  \geq\sum_{i\leq m}\left(  \frac{n-2}{n}-\varepsilon\right)  t_{i}^{2}\geq0,
\end{align*}
where we used (\ref{AT-maxmin}) and we also assumed $\varepsilon\leq\frac
{n-2}{n}.$

Note that for $n-1\geq m\geq1$
\[
1-\frac{4}{\sqrt{4n+1}+1}\leq2-\frac{m}{n}-\sqrt{\left(  1-\frac{m}{n}\right)
^{2}+\frac{4}{n}}\leq\frac{n-2}{n},
\]
therefore we have proved (\ref{jacobi-bm}) with $n-1\geq m\geq1.$ When $m=n,$
we have $\lambda_{1}=\cdots=\lambda_{n}>0.$ Then from (\ref{I}) we see in a
much easier way that (\ref{jacobi-bm}) holds.

The proof of Lemma 2.3 is complete.
\end{proof}

\begin{proposition}
\label{PIJacobi}Let $u$ be a smooth solution to the special\ Lagrangian
equation (\ref{EsLag}) with $n\geq2$ and $\Theta\geq\left(  n-2\right)  \pi/2$
on $B_{R}\left(  0\right)  \subset\mathbb{R}^{n}.$ Set
\[
b=\ln\sqrt{1+\lambda_{\max}^{2}},
\]
where $\lambda_{\max}$ is the largest eigenvalue of Hessian $D^{2}u,$ namely,
$\lambda_{\max}=\lambda_{1}\geq\cdots\geq\lambda_{n}.$ Then $b$ satisfies the
integral Jacobi inequality
\begin{equation}
\int_{B_{R}}-\left\langle \nabla_{g}\varphi,\nabla_{g}b\right\rangle
_{g}dv_{g}\geq\varepsilon\left(  n\right)  \int_{B_{R}}\varphi\left\vert
\nabla_{g}b\right\vert ^{2}dv_{g} \label{IJacobi}%
\end{equation}
for all non-negative $\varphi\in C_{0}^{\infty}\left(  B_{R}\right)  ,$ where
$\varepsilon\left(  n\right)  =1-4/\left(  \sqrt{4n+1}+1\right)  .$
\end{proposition}

\begin{proof}
If $b\left(  x\right)  =b_{1}\left(  x\right)  $ is smooth everywhere, then
the pointwise Jacobi inequality (\ref{jacobi-bm}) in Lemma 2.3 with $m=1$
already implies the integral Jacobi inequality (\ref{IJacobi}).\ In general,
we know that $\lambda_{\max}$ is only a Lipschitz function of the entries of
the Hessian $D^{2}u.$ By the assumption, $D^{2}u\left(  x\right)  $ is smooth
in $x,$ thus $b=b_{1}=\ln\sqrt{1+\lambda_{\max}^{2}}$ is Lipschitz in terms of
$x.$

Set $\varepsilon=\varepsilon\left(  n\right)  .$ We first show that%
\[
\bigtriangleup_{g}b\geq\varepsilon\left\vert \nabla_{g}b\right\vert
^{2}\ \ \ \text{in the viscosity sense.}%
\]
Given any quadratic polynomial $Q$ touching $b$ from above at $p.$ If $p$ is a
smooth point of $b,$ by (\ref{jacobi-bm}) with $m=1,$ we get%
\[
\bigtriangleup_{g}Q\geq\varepsilon\left\vert \nabla_{g}Q\right\vert
^{2}\ \ \ \text{at }p.
\]
Otherwise, eigenvalue $\lambda_{1}$ is not distinct at $p.$ Suppose
$\lambda_{1}=\cdots=\lambda_{k}>\lambda_{k+1}$ at $p.$ Then $Q$ also touches
the smooth $b_{k}=\left(  \ln\sqrt{1+\lambda_{1}^{2}}+\cdots+\ln
\sqrt{1+\lambda_{k}^{2}}\right)  /k$ from above at $p,$ because%
\[
b\left(  x\right)  \geq b_{k}\left(  x\right)  \ \ \text{and }b\left(
p\right)  =b_{k}\left(  p\right)  .
\]
By pointwise Jacobi inequality (\ref{jacobi-bm}) with $m=k,$ we still have%
\[
\bigtriangleup_{g}Q\geq\varepsilon\left\vert \nabla_{g}Q\right\vert
^{2}\ \ \ \text{at }p.
\]

Next we switch to $a=e^{-\varepsilon b}$ and $a_{k}=e^{-\varepsilon b_{k}},$
the above argument leads to%
\[
\bigtriangleup_{g}a\leq0\ \text{in the viscosity sense.}%
\]
Relying on the definition of viscosity supersolutions, we see $a$ is
$\bigtriangleup_{g}$-superharmonic in the potential sense, namely, $a\geq h$
in any regular domain $\Omega$ for $\bigtriangleup_{g}$-harmonic function $h$
with the boundary value $a$ on $\partial\Omega:$%
\[
\left\{
\begin{array}
[c]{c}%
\bigtriangleup_{g}h=0\ \ \ \text{in }\Omega\\
h=a\ \ \ \ \ \text{on }\partial\Omega
\end{array}
\right.  .
\]
By [HH, Theorem 1] (see also [Wn, p. 246]), we obtain%
\[
\bigtriangleup_{g}a\leq0\ \text{in the distribution sense.}%
\]
Note $a$ is Lipschitz because $b$ is. We move to the integral Jacobi
inequality as follows. Take the test function $\varphi e^{\varepsilon b}$ for
and nonnegative $\varphi\in C_{0}^{\infty},$ we get%
\begin{align*}
0  &  \geq\int_{B_{R}}\varphi e^{\varepsilon b}\bigtriangleup_{g}%
a\ dv_{g}=\int_{B_{R}}-\left\langle \nabla_{g}\left(  \varphi e^{\varepsilon
b}\right)  ,\nabla_{g}a\right\rangle _{g}dv_{g}\\
&  =\int_{B_{R}}\ \left\langle e^{\varepsilon b}\left(  \nabla_{g}%
\varphi+\varepsilon\varphi\nabla_{g}b\right)  ,\varepsilon e^{-\varepsilon
b}\nabla_{g}b\right\rangle _{g}dv_{g}\\
&  =\int_{B_{R}}\left(  \varepsilon\left\langle \nabla_{g}\varphi,\nabla
_{g}b\right\rangle _{g}+\varepsilon^{2}\varphi\left\vert \nabla_{g}%
b\right\vert _{g}^{2}\right)  dv_{g}.
\end{align*}
Thus we arrive at the integral Jacobi inequality (\ref{IJacobi}).
\end{proof}

\ 

\section{Proof Of Theorem 1.1}

We assume that $R=2n+1$ and $u$ is a solution on $B_{2n+1}\subset
\mathbb{R}^{n}$ for simplicity of notation. By scaling $v\left(  x\right)
=u\left(  \frac{R}{2n+1}x\right)  /\left(  \frac{R}{2n+1}\right)  ^{2},$ we
still get the estimate in Theorem 1.1. We consider the case $\Theta\geq\left(
n-2\right)  \pi/2.$ The negative phase case $\Theta\leq-\left(  n-2\right)
\pi/2$ follows by symmetry.

Step 1. By the integral Jacobi inequality (\ref{IJacobi}) in Proposition
\ref{PIJacobi}, $b$\ is subharmonic in the integral sense. Then $b^{\frac
{n}{n-2}}$ is also subharmonic in the integral sense on the minimal surface
$\mathfrak{M}=\left(  x,Du\right)  :$%
\begin{align*}
&  \int-\left\langle \nabla_{g}\varphi,\nabla_{g}b^{\frac{n}{n-2}%
}\right\rangle _{g}dv_{g}\\
&  =\int-\left\langle \nabla_{g}\left(  \frac{n}{n-2}b^{\frac{2}{n-2}}%
\varphi\right)  -\frac{2n}{(n-2)^{2}}b^{\frac{4-n}{n-2}}\varphi\nabla
_{g}b,\nabla_{g}b\right\rangle _{g}dv_{g}\\
&  \geq\int\left(  \frac{n}{n-2}\varepsilon(n)\varphi b^{2}\left\vert
\nabla_{g}b\right\vert ^{2}+\frac{2n}{(n-2)^{2}}b^{\frac{4-n}{n-2}}%
\varphi\left\vert \nabla_{g}b\right\vert ^{2}\right)  dv_{g}\geq0
\end{align*}
for all non-negative $\varphi\in C_{0}^{\infty},$ where we approximate $b$ by
smooth functions if necessary.

Applying Michael-Simon's mean value inequality [MS, Theorem 3.4] to the
Lipschitz subharmonic function $b^{\frac{n}{n-2}},$ we obtain
\[
b\left(  0\right)  \leq C\left(  n\right)  \left(  \int_{\mathfrak{B}_{1}%
\cap\mathfrak{M}}b^{\frac{n}{n-2}}dv_{g}\right)  ^{\frac{n-2}{n}}\leq C\left(
n\right)  \left(  \int_{B_{1}}b^{\frac{n}{n-2}}dv_{g}\right)  ^{\frac{n-2}{n}%
},
\]
where $\mathfrak{B}_{r}$ is the ball with radius $r$ and center at $\left(
0,Du\left(  0\right)  \right)  $ in $\mathbb{R}^{n}\times\mathbb{R}^{n}$, and
$B_{r}$ is the ball with radius $r$ and center at $0$ in $\mathbb{R}^{n}.$
Choose a cut-off function $\varphi\in C_{0}^{\infty}\left(  B_{2}\right)  $
such that $\varphi\geq0,$ $\varphi=1$ on $B_{1},$ and $\left\vert
D\varphi\right\vert \leq1.1;$ we then have
\[
\left(  \int_{B_{1}}b^{\frac{n}{n-2}}dv_{g}\right)  ^{\frac{n-2}{n}}%
\leq\left(  \int_{B_{2}}\varphi^{\frac{2n}{n-2}}b^{\frac{n}{n-2}}%
dv_{g}\right)  ^{\frac{n-2}{n}}=\left(  \int_{B_{2}}\left(  \varphi
b^{1/2}\right)  ^{\frac{2n}{n-2}}dv_{g}\right)  ^{\frac{n-2}{n}}.
\]
Applying the Sobolev inequality on the minimal surface $\mathfrak{M}$ [MS,
Theorem 2.1] or [A, Theorem 7.3] to $\varphi b^{1/2},$ which we may assume to
be $C^{1}$ by approximation, we obtain
\[
\left(  \int_{B_{2}}\left(  \varphi b^{1/2}\right)  ^{\frac{2n}{n-2}}%
dv_{g}\right)  ^{\frac{n-2}{n}}\leq C\left(  n\right)  \int_{B_{2}}\left\vert
\nabla_{g}\left(  \varphi b^{1/2}\right)  \right\vert ^{2}dv_{g}.
\]
Decomposing the integrand as follows
\begin{align*}
\left\vert \nabla_{g}\left(  \varphi b^{1/2}\right)  \right\vert ^{2}  &
=\left\vert \frac{1}{2b^{1/2}}\varphi\nabla_{g}b+b^{1/2}\nabla_{g}%
\varphi\right\vert ^{2}\leq\frac{1}{2b}\varphi^{2}\left\vert \nabla
_{g}b\right\vert ^{2}+2b\left\vert \nabla_{g}\varphi\right\vert ^{2}\\
&  \leq\frac{1}{\ln\left(  4/3\right)  }\varphi^{2}\left\vert \nabla
_{g}b\right\vert ^{2}+2b\left\vert \nabla_{g}\varphi\right\vert ^{2},
\end{align*}
where we used
\[
b\geq\ln\sqrt{1+\tan^{2}\left(  \frac{\pi}{2}-\frac{\pi}{n}\right)  }\geq
\ln\sqrt{4/3},
\]
we get
\begin{align*}
b\left(  0\right)   &  \leq C\left(  n\right)  \int_{B_{2}}\left\vert
\nabla_{g}\left(  \varphi b^{1/2}\right)  \right\vert ^{2}dv_{g}\\
&  \leq C\left(  n\right)  \left(  \int_{B_{2}}\varphi^{2}\left\vert
\nabla_{g}b\right\vert ^{2}dv_{g}+\int_{B_{2}}b\left\vert \nabla_{g}%
\varphi\right\vert ^{2}dv_{g}\right)  .
\end{align*}

Step 2. By (\ref{IJacobi}) in Proposition \ref{PIJacobi}, $b$ satisfies the
Jacobi inequality in the integral sense:
\[
\frac{1}{\varepsilon\left(  n\right)  }\bigtriangleup_{g}b\geq\left\vert
\nabla_{g}b\right\vert ^{2}.
\]
Multiplying both sides by the above non-negative cut-off function $\varphi\in
C_{0}^{\infty}\left(  B_{2}\right)  ,$ then integrating, we obtain
\begin{align*}
\int_{B_{2}}\varphi^{2}\left\vert \nabla_{g}b\right\vert ^{2}dv_{g}  &
\leq\frac{1}{\varepsilon\left(  n\right)  }\int_{B_{2}}\varphi^{2}%
\bigtriangleup_{g}b\ dv_{g}\\
&  =\frac{-1}{\varepsilon\left(  n\right)  }\int_{B_{2}}\left\langle
2\varphi\nabla_{g}\varphi,\nabla_{g}b\right\rangle dv_{g}\\
&  \leq\frac{1}{2}\int_{B_{2}}\varphi^{2}\left\vert \nabla_{g}b\right\vert
^{2}dv_{g}+\frac{2}{\varepsilon\left(  n\right)  ^{2}}\int_{B_{2}}\left\vert
\nabla_{g}\varphi\right\vert ^{2}dv_{g}.
\end{align*}
It follows that
\[
\int_{B_{2}}\varphi^{2}\left\vert \nabla_{g}b\right\vert ^{2}dv_{g}\leq
\frac{4}{\varepsilon\left(  n\right)  ^{2}}\int_{B_{2}}\left\vert \nabla
_{g}\varphi\right\vert ^{2}dv_{g}.
\]
So far we have reached%
\begin{align}
b\left(  0\right)   &  \leq C\left(  n\right)  \left(  \int_{B_{2}}\left\vert
\nabla_{g}\varphi\right\vert ^{2}dv_{g}+\int_{B_{2}}b\left\vert \nabla
_{g}\varphi\right\vert ^{2}dv_{g}\right) \nonumber\\
&  \leq C\left(  n\right)  \int_{B_{2}}b\left\vert \nabla_{g}\varphi
\right\vert ^{2}dv_{g}\nonumber\\
&  \leq C\left(  n\right)  \int_{B_{2}}b\sum_{i=1}^{n}\frac{1}{1+\lambda
_{i}^{2}}\sqrt{\det g}\ dx, \label{conformal-b}%
\end{align}
where in the second inequality, we again used $b\geq\ln\sqrt{4/3}.$

Step 3. Differentiating the complex identity%
\begin{align*}
\ln V+\sqrt{-1}\sum_{i=1}^{n}\arctan\lambda_{i}  &  =\ln%
{\textstyle\prod\limits_{i=1}^{n}}
\left(  1+\sqrt{-1}\lambda_{i}\right) \\
&  =\ln\left[  \sum_{0\leq2k\leq n}\left(  -1\right)  ^{k}\sigma_{2k}%
+\sqrt{-1}\sum_{1\leq2k+1\leq n}\left(  -1\right)  ^{k}\sigma_{2k+1}\right]  .
\end{align*}
we obtain the (conformality) identity%
\[
\left(  \frac{1}{1+\lambda_{1}^{2}},\cdots,\frac{1}{1+\lambda_{n}^{2}}\right)
V=\left(  \frac{\partial\Sigma}{\partial\lambda_{1}},\cdots,\frac
{\partial\Sigma}{\partial\lambda_{n}}\right)
\]
with $V=\sqrt{\det g}$ and
\begin{align*}
\Sigma &  =\cos\Theta\sum_{1\leq2k+1\leq n}\left(  -1\right)  ^{k}%
\sigma_{2k+1}-\sin\Theta\sum_{0\leq2k\leq n}\left(  -1\right)  ^{k}\sigma
_{2k}\\
&  =\sigma_{n-1}-\sigma_{n-3}+\cdots,\ \text{in particular when }\left\vert
\Theta\right\vert =\left(  n-2\right)  \frac{\pi}{2}.
\end{align*}
Taking trace, we then get%
\begin{gather}
\sum_{i=1}^{n}\frac{1}{1+\lambda_{i}^{2}}V=\sum_{i=1}^{n}\frac{\partial\Sigma
}{\partial\lambda_{i}}\nonumber\\
=\cos\Theta\sum_{1\leq2k+1\leq n}\left(  -1\right)  ^{k}\left(  n-2k\right)
\sigma_{2k}-\sin\Theta\sum_{0\leq2k\leq n}\left(  -1\right)  ^{k}\left(
n-2k+1\right)  \sigma_{2k-1}\nonumber\\
=c_{0}+c_{1}\sigma_{1}+\cdots+c_{n-1}\sigma_{n-1}, \label{conformal-trace}%
\end{gather}
where the coefficient $c_{i}$ depends only on $i,n,$ and $\Theta.$ At the
critical phase $\left\vert \Theta\right\vert =\left(  n-2\right)  \pi/2,$ the
leading term in (\ref{conformal-trace}) is $\sigma_{n-2}$%
\begin{equation}
\sum_{i=1}^{n}\frac{1}{1+\lambda_{i}^{2}}V=2\sigma_{n-2}-4\sigma_{n-4}+\cdots.
\label{conformal-trace-critical}%
\end{equation}
In turn, (\ref{conformal-b}) becomes%
\begin{equation}
b\left(  0\right)  \leq C\left(  n\right)  \int_{B_{2}}b\left(  c_{0}%
+c_{1}\sigma_{1}+\cdots+c_{n-1}\sigma_{n-1}\right)  dx. \label{b-sigma}%
\end{equation}

Step 4. Next we estimate the integrals $\int b\sigma_{k}dx$ for $1\leq k\leq
n-1$ inductively, using the divergence structure of $\sigma_{k}(D^{2}u):$
\begin{align*}
k\sigma_{k}(D^{2}u)  &  =\sum_{i,j=1}^{n}\frac{\partial\sigma_{k}}{\partial
u_{ij}}\frac{\partial^{2}u}{\partial x_{i}\partial x_{j}}=\sum_{i,j=1}%
^{n}\frac{\partial}{\partial x_{i}}\left(  \frac{\partial\sigma_{k}}{\partial
u_{ij}}\frac{\partial u}{\partial x_{j}}\right) \\
&  =div\left(  L_{\sigma_{k}}Du\right)  ,
\end{align*}
where $L_{\sigma_{k}}$ denotes the matrix $\left(  \frac{\partial\sigma_{k}%
}{\partial u_{ij}}\right)  .$ Let $\psi$ be a smooth cut-off function on
$B_{\rho+1}$ such that $\psi=1$ on $B_{\rho},$ $0\leq\psi\leq1,$ and
$\left\vert D\psi\right\vert \leq1.1.$ Noticing that $\sigma_{k}>0$ by
(\ref{AT-sigma}) in Lemma 2.1 and $b>0,$ we have
\begin{gather}
\int_{B_{\rho}}b\sigma_{k}dx\leq\int_{B_{\rho+1}}\psi b\sigma_{k}%
dx=\int_{B_{\rho+1}}\psi b\frac{1}{k}div\left(  L_{\sigma_{k}}Du\right)
dx\nonumber\\
=\frac{1}{k}\int_{B_{\rho+1}}-\left\langle bD\psi+\psi Db,L_{\sigma_{k}%
}Du\right\rangle dx\nonumber\\
\leq C(n)\Vert Du\Vert_{L^{\infty}(B_{\rho+1})}\left[
\begin{array}
[c]{l}%
\int_{B_{\rho+1}}b\sigma_{k-1}dx+\\
\int_{B_{\rho+1}}\left[  \left\vert \nabla_{g}b\right\vert ^{2}+tr\left(
g^{ij}\right)  \right]  \sqrt{\det g}\ dx
\end{array}
\right]  . \label{Induction-bsigma}%
\end{gather}
The last inequality was derived as follows. As all the above integrands are
invariant under orthogonal transformations, at any point $p\in B_{\rho+1},$ we
assume $D^{2}u\left(  p\right)  $ is diagonalized. Then $L_{\sigma_{k}}$ is
also diagonal with positive entries $\partial_{\lambda_{i}}\sigma_{k}.$ The
positivity can be seen by applying Lemma 2.1 to all $\lambda_{1,}%
\cdots,\lambda_{n}$ but $\lambda_{i},$ whose corresponding phase is no less
than $\left(  n-3\right)  \pi/2.$ Thus $0<\partial_{\lambda_{i}}\sigma
_{k}<(n-k+1)\sigma_{k-1}.$ Now we have%
\begin{align*}
&  \left\vert \left\langle bD\psi+\psi Db,L_{\sigma_{k}}Du\right\rangle
\right\vert \overset{p}{\leq}\sum_{i=1}^{n}\left(  b\left\vert D_{i}%
\psi\right\vert +\psi\left\vert D_{i}b\right\vert \right)  \ \partial
_{\lambda_{i}}\sigma_{k}\ \left\vert D_{i}u\right\vert \\
&  \overset{p}{\leq}C\left(  n\right)  \left\vert Du\left(  p\right)
\right\vert \left(  b\sigma_{k-1}+\sum_{i=1}^{n}\left\vert D_{i}b\right\vert
\partial_{\lambda_{i}}\sigma_{k}\right)  .
\end{align*}
Recall $k\leq n-1,$ then $\partial_{\lambda_{i}}\sigma_{k}$ only consists of
multiples of at most $\left(  n-2\right)  $ eigenvalues without $\lambda_{i}.$
\textquotedblleft Twist\textquotedblright\ multiplying the two $g^{\heartsuit
\heartsuit}$ terms involving the missed $\lambda_{i}$ and the other
eigenvalue, we obtain%
\begin{align*}
&  \left\vert D_{i}b\right\vert \partial_{\lambda_{i}}\sigma_{k}\overset
{p}{\leq}\left\vert D_{i}b\right\vert \partial_{\lambda_{i}}\sigma_{k}\left(
\left\vert \lambda_{1}\right\vert ,\cdots,\left\vert \lambda_{n}\right\vert
\right) \\
&  \ \ \ \ \ \ \ \ \overset{p}{\leq}C\left(  n\right)  \sum_{\alpha\neq
i}\left(  \frac{\left\vert D_{i}b\right\vert ^{2}}{1+\lambda_{i}^{2}}+\frac
{1}{1+\lambda_{\alpha}^{2}}\right)  \sqrt{\left(  1+\lambda_{1}^{2}\right)
\cdots\left(  1+\lambda_{n}^{2}\right)  }.
\end{align*}
Summing up, we get%
\begin{align*}
&  \sum_{i=1}^{n}\left\vert D_{i}b\right\vert \partial_{\lambda_{i}}\sigma
_{k}\overset{p}{\leq}C\left(  n\right)  \sum_{i=1}^{n}\left(  g^{ii}\left\vert
D_{i}b\right\vert ^{2}+g^{ii}\right)  \sqrt{\det g}\\
&  \overset{p}{=}C\left(  n\right)  \left[  \left\vert \nabla_{g}b\right\vert
^{2}+tr\left(  g^{ij}\right)  \right]  \sqrt{\det g}.
\end{align*}
The inequality (\ref{Induction-bsigma}) has been established. To simplify the
last integral in (\ref{Induction-bsigma}), we repeat the integral Jacobi
argument in Step 2 to get%
\[
\int_{B_{\rho+1}}\left\vert \nabla_{g}b\right\vert ^{2}\sqrt{\det g}\ dx\leq
C\left(  n\right)  \int_{B_{\rho+2}}tr\left(  g^{ij}\right)  \sqrt{\det
g}\ dx.
\]
Hence (\ref{Induction-bsigma}) becomes the following inductive inequality%
\begin{equation}
\int_{B_{\rho}}b\sigma_{k}dx\leq C(n)\Vert Du\Vert_{L^{\infty}(B_{\rho+1}%
)}\left[  \int_{B_{\rho+1}}b\sigma_{k-1}dx+\int_{B_{\rho+2}}tr\left(
g^{ij}\right)  \sqrt{\det g}\ dx\right]  . \label{Induction-bsigma-final}%
\end{equation}

Step 4.1. We iterate (\ref{Induction-bsigma-final}) to derive%
\begin{gather*}
\int_{B_{2}}b\sigma_{k}dx\\
\leq C\left(  n\right)  \left\{
\begin{array}
[c]{l}%
\left\Vert Du\right\Vert _{L^{\infty}\left(  B_{2+k}\right)  }^{k}%
\int_{B_{2+k}}b\ dx+\\
\left[  \left\Vert Du\right\Vert _{L^{\infty}\left(  B_{2+k}\right)  }%
^{k}+\cdots+\left\Vert Du\right\Vert _{L^{\infty}\left(  B_{2+k}\right)
}\right]  \int_{B_{2+k+1}}tr\left(  g^{ij}\right)  \sqrt{\det g}\ dx
\end{array}
\right\} \\
\leq C\left(  n\right)  \left\{
\begin{array}
[c]{l}%
\left\Vert Du\right\Vert _{L^{\infty}\left(  B_{2+k}\right)  }^{k+1}+\\
\left[  \left\Vert Du\right\Vert _{L^{\infty}\left(  B_{2+k}\right)  }%
^{k}+\left\Vert Du\right\Vert _{L^{\infty}\left(  B_{2+k}\right)  }\right]
\int_{B_{2+k+1}}tr\left(  g^{ij}\right)  \sqrt{\det g}\ dx
\end{array}
\right\}  ,
\end{gather*}
where for the last inequality, we used Young's inequality and%
\[
\int_{B_{2+k}}b\ dx\leq C\left(  n\right)  \left\Vert Du\right\Vert
_{L^{\infty}\left(  B_{2+k}\right)  },
\]
which follows from%
\[
b=\ln\sqrt{1+\lambda_{\max}^{2}}<\lambda_{\max}<\lambda_{1}+\lambda_{2}%
+\cdots+\lambda_{n}=\bigtriangleup u
\]
by (\ref{AT-allbut1+}) in Lemma 2.1. Putting all the estimates for
$b\sigma_{k}$s in (\ref{b-sigma}) together, we get%
\begin{equation}
b\left(  0\right)  \leq C\left(  n\right)  \left\{
\begin{array}
[c]{l}%
\left\Vert Du\right\Vert _{L^{\infty}\left(  B_{n+1}\right)  }^{n}+\left\Vert
Du\right\Vert _{L^{\infty}\left(  B_{n+1}\right)  }+\\
\left[  \left\Vert Du\right\Vert _{L^{\infty}\left(  B_{n+1}\right)  }%
^{n-1}+\left\Vert Du\right\Vert _{L^{\infty}\left(  B_{n+1}\right)  }\right]
\int_{B_{n+2}}tr\left(  g^{ij}\right)  \sqrt{\det g}\ dx
\end{array}
\right\}  . \label{conformal for b}%
\end{equation}

Step 4.2. We bound the last integral in the above inequality. Relying on the
trace conformality identity (\ref{conformal-trace}), we derive%
\begin{align}
\int_{B_{n+2}}tr\left(  g^{ij}\right)  \sqrt{\det g}\ dx  &  =\int_{B_{n+2}%
}\left(  c_{0}+c_{1}\sigma_{1}+\cdots+c_{n-1}\sigma_{n-1}\right)
dx\label{conformal-trace-integral}\\
&  \leq C\left(  n\right)  \left[  \left\Vert Du\right\Vert _{L^{\infty
}\left(  B_{2n+1}\right)  }^{n-1}+1\right]  ,\nonumber
\end{align}
where for the last inequality, we repeated the iteration integral estimates
for (\ref{Induction-bsigma-final}) in Step 4.1 with $b=1$ (now much simpler)%
\[
\int_{B_{\rho}}\sigma_{k}dx\leq C\left(  n\right)  \left\Vert Du\right\Vert
_{L^{\infty}\left(  B_{\rho+1}\right)  }\int_{B_{\rho+1}}\sigma_{k-1}dx.
\]

Finally from the above estimates (\ref{conformal-trace-integral}) and
(\ref{conformal for b}), we conclude that%
\[
b\left(  0\right)  \leq C\left(  n\right)  \left[  \left\Vert Du\right\Vert
_{L^{\infty}\left(  B_{2n+1}\right)  }^{2n-2}+\left\Vert Du\right\Vert
_{L^{\infty}\left(  B_{2n+1}\right)  }\right]
\]
and after exponentiating%
\[
\left\vert D^{2}u\left(  0\right)  \right\vert \leq C\left(  n\right)
\exp\left[  C\left(  n\right)  \left\Vert Du\right\Vert _{L^{\infty}\left(
B_{2n+1}\right)  }^{2n-2}\right]  .
\]

Note at the critical phase $\Theta=\left(  n-2\right)  \pi/2,$ because of
(\ref{conformal-trace-critical}), the leading term in (\ref{b-sigma}) and
(\ref{conformal-trace-integral}) is $\sigma_{n-2}.$ The iteration integral
estimates in Step 4.1 and 4.2 start from $\sigma_{n-2}.$ Thus we really obtain%
\[
\left\vert D^{2}u\left(  0\right)  \right\vert \leq C\left(  n\right)
\exp\left[  C\left(  n\right)  \left\Vert Du\right\Vert _{L^{\infty}\left(
B_{2n}\right)  }^{2n-4}\right]  .
\]

The proof of Theorem 1.1 is complete.

\end{document}